\documentclass[11pt, oneside]{article}
\usepackage{mathrsfs}
\usepackage{amsfonts}

\usepackage{latexsym,amssymb,amsmath}
\usepackage{latexsym}
\usepackage{amssymb}
\usepackage{amsmath}
\usepackage{amsthm}
\usepackage{indentfirst}
\usepackage{color}
\usepackage{graphicx}
\numberwithin{equation}{section} \allowdisplaybreaks

\newtheorem{theorem}{\color{black}\indent Theorem}[section]
\newtheorem{lemma}{\color{black}\indent Lemma}[section]

\newtheorem{definition}{\color{black}\indent Definition}[section]
\newtheorem{remark}{\color{black}\indent Remark}[section]

\usepackage{amssymb,amsmath}
\begin{document}
\title{\LARGE\bf  Singular Phenomena of Solutions for Nonlinear Diffusion Equations
involving \\$p(x)$-\hbox{Laplacian} Operator
\thanks{
Supported by NSFC (11271154) and the 985 program of Jilin University
.}
\\
\author{ Bin Guo\thanks{Corresponding author\newline\hspace*{6mm}{\it Email
addresses:}~bguo1029@gmail.com(Bin Guo),~wjgao@jlu.edu.cn (Wenjie
Gao).},~~Wenjie Gao
\\
\small{\it{School of Mathematics, Jilin University,} \it{Changchun
130012, PR China}}} }
\date{} \maketitle

{\bf Abstract:} The authors of this paper study singular
phenomena(vanishing and blowing-up in finite time) of solutions to
the homogeneous $\hbox{Dirichlet}$ boundary value problem of
nonlinear diffusion equations involving $p(x)$-\hbox{Laplacian}
operator and a nonlinear source. The authors discuss how the value of the variable exponent $p(x)$ and initial energy(data) affect the properties of solutions. At the same time, we obtain the critical extinction and blow-up exponents of solutions.


\thispagestyle{empty}
\section{Introduction}
Let $\Omega\subset\mathbb{R}^N (N\geqslant1)$ be a bounded simply
connected domain and $0<T<\infty$. Consider the following
quasilinear degenerate parabolic problem:
\begin{equation}
\begin{cases}
u_{t}=\hbox{div}(|\nabla u|^{p(x)-2}\nabla u)+u^{r-2}u,&(x,t)\in Q_{T},\\
u(x,t)=0,&(x,t)\in\Gamma_{T},\\
u(x,0)=u_{0}(x),&x\in\Omega,
\end{cases}
\end{equation}
where $ Q_{T}=\Omega\times(0,T]$, $\Gamma_{T}$ denotes the lateral
boundary of the cylinder $Q_{T}$, It will be assumed throughout the
paper that the exponent $p(x)$ is continuous in $\Omega$ with
logarithmic module of continuity:
\begin{align}
&1<p^{-}=\inf\limits_{x\in \Omega}p(x)\leqslant p(x)
\leqslant p^{+}=\sup\limits_{x\in \Omega}p(x)<\infty,\\
&\forall x\in\Omega,~y\in
\Omega,~|x-y|<1,~|p(x)-p(y)|\leqslant\omega(|x-y|),
\end{align}
where
\begin{equation*}
\limsup\limits_{\tau\rightarrow0^{+}}\omega(\tau)\ln\frac{1}{\tau}=C<+\infty.
\end{equation*}

Problem $(1.1)$ occurs in mathematical models of physical processes,
for example, nonlinear diffusion, filtration, elastic mechanics and
electro-rheological fluids, the readers may refer to
\cite{EDIBEBN,JRPHILIP,CATCWJ,MRUZICKA,LDPHPM}.
 When $p$ is a fixed constant, the authors in \cite{JXYCHJ}
 discussed the
 extinction and non-extinction of solutions by applying a comparison
 theorem and energy estimate methods. Besides, in \cite{LWJWMX},
 the authors
 studied blowing-up of solutions with positive initial energy.
 However,
 we point out that the methods used in \cite{JXYCHJ,LWJWMX}
 fail in solving our
 problems. The main reason is that
\begin{align*}
&\|\nabla u\|^{r}_{p(.),\Omega}\not\equiv\Big[\int_{\Omega}|\nabla
u|^{p(x)}dx\Big]^{\frac{r}{p(.)}};\\
&\int_{\Omega}u^{m}|\nabla
u|^{p(x)}dx\not\equiv\int_{\Omega}\Big(\frac{p(.)}{m+p(.)}\Big)^{p(.)}|\nabla
u^{\frac{m+p(.)}{p}}|^{p(x)}dx;\\
&\hbox{div}(|\nabla(\lambda u)|^{p(x)-2}\nabla(\lambda u))\not\equiv
\lambda^{p(x)-1}\hbox{div}(|\nabla u|^{p(x)-2}\nabla u).
\end{align*}
Due to the lack of homogeneity, we have to look for new methods or
 techniques to study properties of solutions to the problem. Fortunately, we
 construct a new control function and apply suitable embedding
 theorems to prove that the solution blows up in finite time when the
 initial energy is positive, which improves the
 result in \cite{SNANTO2}. Subsequently, we find that the solution
 represents
 different properties when $p(x)$
 belongs to different intervals or when the initial data is
 sufficiently small or strictly bigger than zero. As we know, such results are seldom seen for the problem with variable exponents. By applying energy estimate method and comparison principle for ODE, we prove that the solution of Problem $(1.1)$ develops a nonempty set $\{x\in\Omega,~u(x,t)=0\}$, the so called dead core, after finite time, or remains positive when $p(x)$
 belongs to different intervals.

The outline of this paper is the following: In Section 2, we shall
introduce the function spaces of $\rm{Orlicz-Sobolev}$ type, give
the definition of the weak solution to the problem and prove that
the weak solution blows up in finite time for a positive initial
energy; Section 3 will be devoted to studying the critical
extinction exponent.
\section{Critical Blow-up exponent}
 In this section, we will study the blowing-up of the weak solutions
 when the initial energy is less than a positive constant.
 Let us introduce the Banach spaces
\[
\begin{split}
& L^{p(x)}(\Omega)=\left\{u(x)|u~is ~measurable~ in
~\Omega,~A_{p(.)}(u)=\int_{\Omega}|u|^{p(x)}dx<\infty\;\right\},
\\
&\|u\|_{p(.)}=\inf\{\lambda>0,A_{p(.)}(u/\lambda)\leqslant1\};
\\
& W^{1,p(x)}(\Omega):=\{u:~u\in L^{p(x)}(\Omega),|\nabla u|\in L^{p(x)}(\Omega)\};
\\
&\|u\|_{W^{1,p(x)}(\Omega)}=\|u\|_{p(.),~\Omega}
+\|\nabla u\|_{p(.),~\Omega};
\\
& V(\Omega)=\left\{u|\,u\in L^{2}(\Omega)\cap
W^{1,1}_{0}(\Omega), \,u\in W^{1,p(x)}(\Omega)\;\right\},
\\
&\|u\|_{V(\Omega)}=\|u\|_{2,~\Omega}
+\|\nabla u\|_{p(.),~\Omega};\\
& H(Q_T)=\left\{u:[0,T]\mapsto V(\Omega)|\,u\in
L^{2}(Q_{T}),~|\nabla u|\in L^{p(x)}(Q_T),\;\mbox{$u=0$ on
$\Gamma_T$}\right\},
\\
& \|u\|_{H(Q_T)}=\|u\|_{2,~Q_{T}} +\|\nabla u\|_{p(.),~Q_T},
\end{split}
\]
and denote by $H'(Q_T)$ the dual of $H(Q_T)$ with
respect to the inner product in $L^{2}(Q_T)$. From \cite{LDPHPM}, we know that
Condition (1.3) can imply that $W^{1,p(x)}_{0}(\Omega):=\{u:~u\in W^{1,p(x)}(\Omega),u=0 ~on~
\partial\Omega\}$ is the
closure of $C^{\infty}_{0}(\Omega)$ in $W^{1,p(x)}(\Omega)$.
\begin{definition}
A function $u(x,t)\in H(Q_T)\cap
L^{\infty}(0,T;L^{2}(\Omega)),u_{t}\in H'(Q_T)$ is called a weak solution of
Problem $(1.1)$ if for every test-function $$\xi\in \mathcal
{Z}\equiv\{\eta(z):\eta\in H(Q_T)\cap
L^{\infty}(0,T;\,L^{2}(\Omega)),\eta_{t}\in H'(Q_T)\},$$
and every $t_{1},t_{2}\in[0,T]$ the following identity holds:
\begin{equation}
\begin{split}
&\int^{t_{2}}_{t_{1}}\int_{\Omega}[u\xi_{t}-|\nabla
u|^{p(x)-2}\nabla u\nabla \xi+u^{r-2}u\xi]dxdt=\int_{\Omega}u\xi
dx\Big|^{t_{2}}_{t_{1}}.
\end{split}
\end{equation}
\end{definition}
For the existence of solutions to Problem $(1.1)$, we have the following theorem
\begin{theorem}{\rm\makeatletter
\def\@cite#1#2{\textsuperscript{[{#1\if@tempswa , #2\fi}]}}
\makeatother\cite{SNANTO1,SZLWJG}}
Suppose that Conditions $(1.2)-(1.3)$ are fulfilled. Then for every $u_{0}\in W^{1,p(x)}_{0}(\Omega)\cap L^{\infty}(\Omega)$, there exists a $T^*>0$ such that Problem $(1.1)$ has at least one weak solution $u\in H(Q_{T^*}),u_{t}\in H'(Q_{T^*})$ in the sense of Definition $(2.1)$.

\end{theorem}
Define $$E(t)=\int_{\Omega}\frac{1}{p(x)}|\nabla
u|^{p(x)}dx-\frac{1}{r}\int_{\Omega}|u|^{r}dx.$$

For the sake of simplicity,  we give some notations used below. By Corollary 3.34 in \cite{LDPHPM}, we know that $W^{1,p(x)}_{0}(\Omega)\hookrightarrow L^{r}(\Omega)(1<r<\frac{Np^{-}}{N-p^{-}}).$ Let
$B$ be the constant of the embedding inequality
$$\|u\|_{r}\leqslant B\|\nabla u\|_{p(.)},~\forall ~u\in W^{1,p(x)}_{0}(\Omega)
.$$  Set
$E_{1}=(\frac{r-p^{+}}{rp^{+}})B_{1}^{\frac{r-p^{+}}{rp^{+}}},\alpha_{1}=B_{1}^{\frac{rp^{+}}{p^{+}-r}}
,$ where  $B_{1}=\max\{B,1\}.$ Our main result is
\begin{theorem}
Assume that $p(x)$ satisfies $(1.2)-(1.3)$ and the following
conditions hold
\begin{align*}
&(H_{1})~u_{0}\in L^{2}(\Omega)\cap W^{1,p(x)}_{0}(\Omega),~E(0)<E_{1},~\min\{|\nabla u_{0}|_{p(x)}^{p^{-}},|\nabla u_{0}|_{p(x)}^{p^{+}}\}>\alpha_{1};\\
&(H_{2})~\max\{1,\frac{2N}{N+2}\}<p^{-}<N,~\max\{2,p^{+}\}<r\leqslant\frac{2N+(N+2)(p^{-}-1)}{N},
\end{align*}
then the solution of Problem $(1.1)$ blows up in finite time.
\end{theorem}
In order to prove this theorem, we first give some lemmas.
\begin{lemma}Suppose that $u(x,t)\in H(Q_T)\cap
L^{\infty}(0,T;L^{2}(\Omega)),u_{t}\in H'(Q_T)$ is a weak solution of Problem $(1.1)$ and $2<r\leqslant\frac{2N+(N+2)(p^{-}-1)}{N}$, then the following conclusions hold

\begin{align*}
&~(i)~u_{t}\in L^{2}(Q_{T}),|\nabla u|\in L^{\infty}(0,T;L^{p(x)}(\Omega));\\
&~(ii)~u\in C(0,T;L^{r}(\Omega)),~~\int_{\Omega}\frac{1}{p(x)}|\nabla u|^{p(x)}dx\in C(0,T);\\
&~(iii)~E(t)\in C[0,T]\cap C^{1}(0,T);\\
&~(iv)~E(t)~\hbox{is non-increasing with respect to t and satisfies the following identity} \\
&~~~~~~~~~~E'(t)=-\|u_{t}\|^{2}_{2}\leqslant0.
\end{align*}
\end{lemma}
\begin{proof}
A weak solution $u(x,t)$ to Problem $(1.1)$ is a limit function of the sequence of $\hbox{Galerkin's}$ approximation
$$u^{(m)}=\sum\limits_{k=1}^{m}c_{k}^{(m)}\varphi_{k},~~\varphi_{k}\in W^{1,p^{+}}_{0}(\Omega),~c_{k}^{(m)}\in C^{1}(0,T).$$
Following the lines of the proof of Lemma 3.1 and Theorem 6.1 in \cite{SNANTO1,SNANTO2}, we know that there exists a positive constant $C=C( |\Omega|,|u_{0}|_{L^{\infty}(\Omega)},p^{\pm},r,N)$ such that
\begin{align}
&~\|u^{(m)}\|_{H(Q_{T})}+\|u^{(m)}\|_{L^{\infty}(0,T;L^{2}(\Omega))}+\|u_{t}^{(m)}\|_{H'(Q_{T})}\leqslant C ;\\
&~\|u_{t}^{(m)}\|_{L^{2}(\Omega)}+\frac{d}{dt}\int_{\Omega}\frac{1}{p(x)}|\nabla u^{(m)}|^{p(x)}dx=\int_{\Omega}|u^{(m)}|^{r-2}u^{(m)}u^{(m)}_{t}dx.
\end{align}
Proposition 3.1 in \cite{EDIBEBN} and Inequality $(2.2)$ yield
\begin{align}
\|u^{(m)}\|_{L^{\frac{p^{-}(N+2)}{N}}(Q_{T})}\leqslant \gamma\iint_{Q_{T}}|\nabla u^{(m)}|^{p(x)}dxdt\cdot\Big(\sup\limits_{0<t<T}\int_{\Omega}|u^{(m)}|^{2}dx\Big)^{\frac{p^{-}}{N}}\leqslant C.
\end{align}
Furthermore, according to $1<r\leqslant\frac{2N+(N+2)(p^{-}-1)}{N}$ and $(2.4)$, it is easy to verify that
\begin{align}
\||u^{(m)}|^{r-2}u^{(m)}\|_{H(Q_{T})}\leqslant C.
\end{align}

By $(2.2),(2.3),(2.5)$, we have
\begin{align*}
~\|u_{t}^{(m)}\|_{L^{2}(Q_{T}}+\int_{\Omega}\frac{1}{p(x)}|\nabla u^{(m)}|^{p(x)}dx\leqslant C:=C(p^{\pm},|\nabla u_{0}(\Omega)|_{p(.)},|\Omega|),
\end{align*}
which implies $u(x,t)\in L^{2}(Q_{T}),|\nabla u|\in L^{\infty}(0,T;L^{p(x)}(\Omega)).$

Noting that $W^{1,p(x)}_{0}(\Omega)\hookrightarrow W^{1,p^{-}}_{0}(\Omega)\overset{compact}\hookrightarrow L^{r}(\Omega)\hookrightarrow L^{2}(\Omega)$ and applying Corollary 6 in \cite{JSIMON}, we get
$u\in C(0,T;L^{r}(\Omega)).$

Similarly as the proof of Lemma 1 in \cite{SNANTO2}, we have
\begin{align}
\|u_{t}\|_{L^{2}(\Omega\times(t_{1},t_{2}))}+E(t_{2})=E(t_{1}),~~0\leqslant t_{1}<t_{2}\leqslant T,
\end{align}
which shows $E(t)\in C[0,T],~\int_{\Omega}\frac{1}{p(x)}|\nabla u|^{p(x)}dx \in C(0,T)$ from absolute continuity of $\hbox{Lebesgue}$ measure.

Letting $t_{1}=t,~t_{2}=t+h,~t,t+h \in(0,T)$, multiplying (2.6) by $\frac{1}{h}$ and according to $|u_{t}|_{L^{2}(\Omega)}\in L^{2}(0,T)$ and $\hbox{Lebesgue}$ differentiation theorem, we have
\begin{align*}
E'(t)
=-\int_{\Omega}|u_{t}|^{2}dx\leqslant0,
\end{align*}
that is $E(t)\in C^{1}(0,T).$
\end{proof}
\begin{lemma}
Suppose that $u$ is the solution of Problem $(1.1)$. If the
condition $(H_{1})$ holds and $r>\max\{2,p^{+}\}$, then there exists a positive constant
$\alpha_{2}>\alpha_{1}$ such that for all $t\geqslant0$
\begin{align}
&\int_{\Omega}|\nabla u|^{p(x)}dx\geqslant\alpha_{2},\\
&\int_{\Omega}|u|^{r}dx\geqslant
B_{1}^{r}\max\{\alpha_{2}^{\frac{r}{p^{-}}},\alpha_{2}^{\frac{r}{p^{+}}}\}.
\end{align}
\end{lemma}
\begin{proof}
\begin{equation}
\begin{split}
E(t)&\geqslant\frac{1}{p^{+}}\int_{\Omega}|\nabla
u|^{p(x)}dx-\frac{B^{r}}{r}\|\nabla u\|^{r}_{p(.)}\\
&\geqslant\frac{1}{p^{+}}\int_{\Omega}|\nabla u|^{p(x)}dx
-\frac{B^{r}}{r}\max\Big\{(\int_{\Omega}|\nabla
u|^{p(x)}dx)^{\frac{1}{p^{-}}},(\int_{\Omega}|\nabla
u|^{p(x)}dx)^{\frac{1}{p^{+}}}\Big\}^{r}\\
&\geqslant\frac{1}{p^{+}}\int_{\Omega}|\nabla u|^{p(x)}dx
-\frac{B_{1}^{r}}{r}\max\Big\{(\int_{\Omega}|\nabla
u|^{p(x)}dx)^{\frac{r}{p^{-}}},(\int_{\Omega}|\nabla
u|^{p(x)}dx)^{\frac{r}{p^{+}}}\Big\}\\
&\overset{\Delta}=\frac{1}{p^{+}}\alpha
-\frac{B_{1}^{r}}{r}\max\{\alpha^{\frac{r}{p^{-}}},\alpha^{\frac{r}{p^{+}}}\}=h(\alpha),
\end{split}
\end{equation}
with $\alpha=\int_{\Omega}|\nabla u|^{p(x)}dx.$

Next, we will give a simple analysis about the properties of the function $h(\alpha).$ It is easy to prove that $h(\alpha)$ satisfies the following properties
\begin{equation}\begin{split}
&~h(\alpha)\in C[0,+\infty);\\
&~h'(\alpha)=\begin{cases}\frac{1}{p^{+}}-\frac{B_{1}^{r}}{p^{-}}\alpha^{\frac{r-p^{-}}{p^{-}}}<0,~~&\alpha>1;\\
\frac{1}{p^{+}}-\frac{B_{1}^{r}}{p^{+}}\alpha^{\frac{r-p^{+}}{p^{+}}},~~&\alpha<1;\end{cases}\\
&~h'_{+}(1)=\frac{1}{p^{+}}-\frac{B_{1}^{r}}{p^{-}}<0,~h'_{-}(1)=\frac{1}{p^{+}}-\frac{B_{1}^{r}}{p^{+}}<0;\\
&h'(\alpha_{1})=0,~~0<\alpha_{1}<1.
\end{split}\end{equation}

Although the function $h(\alpha)$ is not differentiable at $\alpha=1,$  a simple analysis shows that
$h(\alpha)$ is increasing for $0<\alpha<\alpha_{1}$ while
$h(\alpha)$ is decreasing for $\alpha\geqslant\alpha_{1}$, and
$\lim\limits_{\alpha\rightarrow\infty}h(\alpha)=-\infty.$ Due to
$E(0)<E_{1},$ then there exists a positive constant
$\alpha_{2}>\alpha_{1}$ such that
$h(\alpha_{1})=E(0).$  By $\min\{|\nabla u_{0}|_{p(x)}^{p^{-}},|\nabla u_{0}|_{p(x)}^{p^{+}}\}>\alpha_{1},$
 we get $$h(\alpha_{0})\leqslant E(0)=h(\alpha_{2}),$$
where $\alpha_{0}=\int_{\Omega}|\nabla u_{0}|^{p(x)}dx.$ Once again
 applying the monotonicity of $h(\alpha)$, we have $\alpha_{0}\geqslant\alpha_{2}.$

We prove $(2.7)$ by arguing by contradiction. Suppose that there
exists a $t_{0}>0$ such that $
\int_{\Omega}|\nabla u(.,t_{0})|^{p(x)}dx<\alpha_{2}.$ Since $\int_{\Omega}\frac{1}{p(x)}|\nabla u|^{p(x)}dx\in C(0,T)$, we may choose a $t_{1}>0$ such that
$$\alpha_{2}>\int_{\Omega}|\nabla u(.,t_{1})|^{p(x)}dx\geqslant p^{-}\int_{\Omega}\frac{1}{p(x)}|\nabla u(.,t_{1})|^{p(x)}dx>\alpha_{1}.$$
By the definitions of $E(t)$ and the monotonicity of $h(\alpha)$, we have
\begin{align*}
E(t_{1})> h\Big(\int_{\Omega}|\nabla u(.,t_{1})|^{p(x)}dx\Big)\geqslant h(\alpha_{2})=E(0),
\end{align*}
which
contradicts $E(t)\leqslant E(0),~\forall ~t\geqslant0.$

Noting that $E'(t)\leqslant0,$  we get
\begin{align*}
\frac{1}{r}\int_{\Omega}|u|^{r}dx&\geqslant
\frac{1}{p^{+}}\int_{\Omega}|\nabla u|^{p(x)}dx-E(0)
=\frac{1}{p^{+}}\int_{\Omega}|\nabla u|^{p(x)}dx-h(\alpha_{2})\\
&=\frac{1}{p^{+}}\int_{\Omega}|\nabla
u|^{p(x)}dx-\frac{1}{p^{+}}\alpha_{2}+
\frac{B_{1}^{r}}{r}\max\{\alpha_{2}^{\frac{r}{p^{-}}},\alpha_{2}^{\frac{r}{p^{+}}}\}
\geqslant
\frac{B_{1}^{r}}{r}\max\{\alpha_{2}^{\frac{r}{p^{-}}},\alpha_{2}^{\frac{r}{p^{+}}}\}.
\end{align*}
\end{proof}

Let $H(t)=E_{1}-E(t),$  then
\begin{lemma}
For all $t\geqslant0$, we have
\begin{align}
0<H(0)\leqslant H(t)\leqslant\frac{1}{r}\int_{\Omega}|u|^{r}dx.
\end{align}
\end{lemma}
\begin{proof}
Since $E'(t)\leqslant0,$ it is very easily seen that
$H'(t)\geqslant0$, which shows that $H(t)\geqslant
H(0)=E_{1}-E(0)>0.$ Moreover, a simple calculation shows that
\begin{align*}
E_{1}-\int_{\Omega}\frac{1}{p(x)}|\nabla u|^{p(x)}dx&\leqslant
(\frac{r-p^{+}}{rp^{+}})B_{1}^{\frac{r-p^{+}}{rp^{+}}}-\frac{1}{p^{+}}\alpha_{2}\leqslant
h(\alpha_{1})-\frac{1}{p^{+}}\alpha_{1}<0.
\end{align*}
So $$H(t)=E_{1}-\int_{\Omega}\frac{1}{p(x)}|\nabla
u|^{p(x)}dx+\frac{1}{r}\int_{\Omega}|u|^{r}dx\leqslant\frac{1}{r}\int_{\Omega}|u|^{r}dx.$$
\end{proof}
 {\bf Proof of Theorem 2.2.} Letting
 $G(t)=\frac{1}{2}\int_{\Omega}|u|^{2}dx,$ we have
\begin{equation}
\begin{split}
G'(t)&=\int_{\Omega}u u_{t}dx
=-\int_{\Omega}|\nabla u|^{p(x)}dx+\int_{\Omega}|u|^{r}dx
=-\int_{\Omega}p(x)\frac{1}{p(x)}|\nabla u|^{p(x)}dx+\int_{\Omega}|u|^{r}dx\\
&\geqslant -p^{+}\int_{\Omega}\frac{1}{p(x)}|\nabla
u|^{p(x)}dx+\int_{\Omega}|u|^{r}dx
=\int_{\Omega}|u|^{r}dx-p^{+}(E(t)+\frac{1}{r}\int_{\Omega}|u|^{r}dx)\\
&=(1-\frac{p^{+}}{r})\int_{\Omega}|u|^{r}dx-p^{+}E(t)
\geqslant(1-\frac{p^{+}}{r})\int_{\Omega}|u|^{r}dx-p^{+}E_{1}+p^{+}H(t)\\
&
\geqslant(1-\frac{p^{+}}{r})\int_{\Omega}|u|^{r}dx-p^{+}E_{1}.
\end{split}
\end{equation}
Inequality $(2.8)$ shows that
\begin{equation}
\begin{split}
p^{+}E_{1}=\frac{p^{+}E_{1}}{B_{1}^{r}\max\{\alpha_{2}^{\frac{r}{p^{-}}},\alpha_{2}^{\frac{r}{p^{+}}}\}}
\Big(B_{1}^{r}\max\{\alpha_{2}^{\frac{r}{p^{-}}},\alpha_{2}^{\frac{r}{p^{+}}}\}\Big)
\leqslant\frac{p^{+}E_{1}}{B_{1}^{r}\max\{\alpha_{2}^{\frac{r}{p^{-}}},\alpha_{2}^{\frac{r}{p^{+}}}\}}\int_{\Omega}|u|^{r}dx.
\end{split}
\end{equation}Moreover, $r>2$ and $\textrm{H\"{o}lder's}$ inequality imply that
\begin{align}
\int_{\Omega}|u|^{r}dx
\geqslant|\Omega|^{\frac{2-r}{2}}(\int_{\Omega}|u|^{2}dx)^{\frac{r}{2}}.
\end{align}
So, using $(2.12)-(2.14)$, we get
\begin{equation}G'(t)\geqslant
C_{0}\Big(\int_{\Omega}|u|^{2}dx\Big)^{\frac{r}{2}}=C_{0}G^{\frac{r}{2}}(t),
\end{equation}
where
$$C_{0}=\frac{(r-p^{+})\Big[B_{1}^{r}\max\{\alpha_{2}^{\frac{r}{p^{-}}},\alpha_{2}^{\frac{r}{p^{+}}}\}-B_{1}^{\frac{r-p^{+}}{rp^{+}}}\Big]}
{B_{1}^{r}\max\{\alpha_{2}^{\frac{r}{p^{-}}},\alpha_{2}^{\frac{r}{p^{+}}}\}r}|\Omega|^{\frac{2-r}{2}}2^{\frac{r}{2}}>0$$

Integrating $(2.15)$ with respect to $t$ over $(0,\tau)$,  we
have
\begin{align*}
G(\tau)\geqslant\Big(G^{1-\frac{r}{2}}(0)-(\frac{r}{2}-1)C_{0}\tau\Big)^{\frac{2}{2-r}}.
\end{align*}
Applying $\textrm{Gronwall's}$ inequality, we know that $G(t)$ blows
up in a finite time
$T_{*}\leqslant\frac{G^{1-\frac{r}{2}}(0)}{(\frac{r}{2}-1)C_{0}}$.

For
$\frac{2N}{N+2}<p^{-}<p^{+}<2,r=2,~E(0)<E_{1}$,
what happens to the solution of Problem $(1.1)$? The following
theorem gives a positive answer
\begin{theorem}
Suppose that $p(x)$satisfies $(1.2)-(1.3)$ and the following
conditions hold
\begin{align*}
&(H_{5})~u_{0}\in L^{2}(\Omega)\cap
W^{1,p(x)}_{0}(\Omega),~E(0)<E_{1},
~\min\{|\nabla u_{0}|_{p(x)}^{p^{-}},|\nabla u_{0}|_{p(x)}^{p^{+}}\}>\alpha_{1},\\
&(H_{6})~\frac{2N}{N+2}<p^{-}<p^{+}<r=2,
\end{align*}
then the solution of Problem $(1.1)$ exists globally. Furthermore, we have
$$\lim\limits_{t\rightarrow+\infty}\|u\|_{L^{2}(\Omega)}=+\infty.$$
\end{theorem}
\begin{proof}
By $(2.15)$, we can easily obtain that
$$G'(t)\geqslant C_{0}2^{\frac{2-r}{2}}|\Omega|^{\frac{r-2}{2}}G(t),~~t\geqslant0.$$
Moreover, by applying $\textrm{Gronwall's}$ inequality, we get
$$\lim\limits_{t\rightarrow+\infty}\|u\|_{L^{2}(\Omega)}=+\infty.$$
This completes the proof of this theorem.
\end{proof}
 For $p^{+}<r<2$, we have the following theorem
\begin{theorem}
Suppose that $p(x)$satisfies $(1.2)-(1.3)$ and the following
conditions hold
\begin{align*}
&(H_{7})~u_{0}\in L^{\infty}(\Omega)\cap L^{2}(\Omega)\cap W^{1,p(x)}_{0}(\Omega),~E(0)\leqslant 0;\\
&(H_{8})~\frac{2N}{N+2}<p^{-}<p^{+}<r<2,
\end{align*}
then the nonnegative solution of Problem $(1.1)$ exists globally. Furthermore, we have
$$\lim\limits_{t\rightarrow+\infty}\|u\|_{L^{\infty}(\Omega)}=+\infty.$$
\end{theorem}
\begin{proof} We use a trick used in \cite{SNANTO1,BGWJGAO}.
The function $u^{2k-1}$($k\in \mathbb{N}$) can be chosen as a
test-function in $(2.1)$. In (2.1), let $t_2=t+h,~t_1=t$, with $t,
t+h\in(0, T)$, then
\begin{equation}
\begin{split}
&\displaystyle\frac{1}{2
k}\int_{t}^{t+h}\frac{d}{dt}(\int_{\Omega}u^{2k}dx)dt+\int_{t}^{t+h}\int_{\Omega}(2k-1)
u^{2(k-1)}|\nabla u|^{p(x)}dxdt =\int_{t}^{t+h}\int_{\Omega}
u^{2k-2+r}dxdt.
\end{split}
\end{equation}
Dividing the last equality by $h$, letting $h\rightarrow 0$ and
applying $\hbox{Lebesgue}$ differentiation theorem, we have
that $\forall~t\in(0,T)$
\begin{equation}
\begin{split}
&\displaystyle\frac{1}{2
k}\frac{d}{dt}\int_{\Omega}u^{2k}dx+\int_{\Omega}(2k-1)
u^{2(k-1)}|\nabla u|^{p(x)}dx=\int_{\Omega}u^{2k-2+r}dx.
\end{split}
\end{equation}
By H$\rm{\ddot{o}}lder$'s inequality, we get
\begin{align}|\int_{\Omega}u^{2k-2+r}dx|\leqslant ||u(\cdot, t)||_{L^{2k}(\Omega)}^{2k-2+r}\cdot
|\Omega|^{\frac{2-r}{2k}},~~k=1,2,\cdots.\end{align}

Combing $\textrm{Gronwall's}$ inequality with inequalities
$(2.17)-(2.18)$ and dropping the nonnegative terms, we have
\begin{align}
\|u\|_{L^{2k}(\Omega)}\leqslant\Big(\|u_{0}\|^{2-r}_{L^{2k}(\Omega)}+(1-\frac{r}{2})t|\Omega|^{\frac{2-r}{2k}}\Big)^{\frac{1}{2-r}}.
\end{align}
In (2.19), letting $k\rightarrow \infty$, we have
\begin{align}
\|u\|_{L^{\infty}(\Omega)}\leqslant\Big(\|u_{0}\|^{2-r}_{L^{\infty}(\Omega)}+(1-\frac{r}{2})t\Big)^{\frac{1}{2-r}},~t\geqslant0,
\end{align}
which implies that $\|u\|_{L^{\infty}(\Omega)}$ can not blow up at any finite time. We now prove that
$$\lim\limits_{t\rightarrow+\infty}\|u\|_{L^{\infty}(\Omega)}=+\infty.$$
 If not, there exists a positive constant $M_{0}$ such that
$$\|u\|_{L^{\infty}(\Omega)}\leqslant M_{0},~~~~t\geqslant0.$$
Then, \begin{align} \int_{\Omega}|u|^{2}dx\leqslant
M_{0}^{r-2}\int_{\Omega}|u|^{r}dx.
\end{align}
Moreover, we apply Lemma 2.1 and Inequality $(2.21)$ to obtain
\begin{align*}
G'(t)&=\int_{\Omega}u u_{t}dx=-\int_{\Omega}|\nabla u|^{p(x)}dx+\int_{\Omega}|u|^{r}dx\\
&\geqslant -p^{+}\int_{\Omega}\frac{1}{p(x)}|\nabla
u|^{p(x)}dx+\int_{\Omega}|u|^{r}dx
=\int_{\Omega}|u|^{r}dx-p^{+}(E(t)+\frac{1}{r}\int_{\Omega}|u|^{r}dx)\\
&=(1-\frac{p^{+}}{r})\int_{\Omega}|u|^{r}dx-p^{+}E(t) \geqslant
C_{0}2^{\frac{-r}{2}}|\Omega|^{\frac{r-2}{2}}\int_{\Omega}|u|^{r}dx\\&\geqslant(1-\frac{p^{+}}{r})M^{2-r}_{0}\int_{\Omega}|u|^{2}dx
=C_{0}2^{\frac{2-r}{2}}|\Omega|^{\frac{r-2}{2}}M^{2-r}_{0} G(t),
\end{align*}
which shows that
$$\lim\limits_{t\rightarrow\infty}\|u\|_{L^{2}(\Omega)}=+\infty.$$
This is a contradiction. This completes the proof of this theorem.
\end{proof}

\section{Critical extinction exponent}
In this section, we are devoted to the discussion of the critical
extinction exponent of solutions to Problem $(1.1).$  Namely, we
mainly discuss how the ranges of $p^{+},p^{-}$ and the value of the
initial data $u_{0}$ affect the extinction property of solutions.
\begin{theorem}
Suppose that $p(x)$satisfies $(1.2)-(1.3).$  If the following
condition holds
\begin{align*}
(H_{9})~\frac{2N}{N+2}<p^{-}<p^{+}<r\leqslant2,
\end{align*}
then the nonnegative solution of Problem $(1.1)$ vanishes in finite
time for any nonnegative sufficiently, but small initial data $u_{0}(x)$.
More precise speaking, we have the following estimates
\begin{equation*}
\begin{cases}\|u\|_{2}\leqslant g(t)^{\frac{1}{2-p^{+}}},~~&0<t<T_{1},\\
\|u\|_{2}=0,~&t\in[T_{1},\infty),
\end{cases}
\end{equation*}
where $g(t),~T_{1}$ satisfy
\begin{equation*}
g(t)=\begin{cases}
\|u_{0}\|^{2-p^{+}}_{2}-K_{1}+K_{1}e^{(p^{-}-2)t},~r=2,\\
\|u_{0}\|^{2-p^{+}}_{2}+F(u_{0})t,~~1<r<2;
\end{cases}
\end{equation*}
\begin{equation*}
T_{1}=\begin{cases}\frac{1}{p^{-}-2}\ln(1-\frac{\|u_{0}\|^{2-p^{+}}_{2}}{K_{1}}),
K_{1}=\frac{2-p^{+}}{2-p^{-}}C_{1}\min\{1,\|u_{0}\|^{p^{-}-p^{+}}_{2}\},~~r=2;\\
\frac{\|u_{0}\|^{2-p^{+}}_{2}}{-F(u_{0})},~F(u_{0})=
(2-p^{+})\Big[2|\Omega|^{\frac{2-r}{2}}\|u_{0}\|^{r-p^{+}}_{2}-
\frac{C_{1}}{2}\min\{\|u_{0}\|^{p^{-}-p^{+}}_{2},1\}\Big],
~1<r<2.\end{cases}
\end{equation*}
Here $C_{1}$ is a positive constant.
\end{theorem}
\begin{proof}
Multiplying the first equation in $(1.1)$ by $u$ and  integrating
over $\Omega\times(t,t+h)$, we have
\begin{align}
\frac{1}{2}\int_{\Omega}u^{2}dx\Big|^{t+h}_{t}+\int_{t}^{t+h}\int_{\Omega}|\nabla
u|^{p(x)}dxd\tau=\int_{t}^{t+h}\int_{\Omega}u^{r}dxd\tau.
\end{align}

Dividing $(3.1)$ by $h$ and applying $\rm{Lebesgue}$ differentiation
theorem, we have
\begin{align}
G'(t)+\int_{\Omega}|\nabla
u|^{p(x)}dx\leqslant2\int_{\Omega}|u|^{r}dx,
\end{align}
where $G(t)=\int_{\Omega}u^{2}dx.$

First we consider the case when $r=2.$ By means of the embedding theorem $W^{1,p(x)}_{0}(\Omega)\hookrightarrow
W^{1,p^{-}}_{0}(\Omega)\hookrightarrow L^{2}(\Omega),$ we have
\begin{equation}
\begin{split}
\|u\|_{2}&\leqslant C\|\nabla u\|_{p^{-}}\leqslant C\|\nabla
u\|_{p(.)}\leqslant C\max\Big[(\int_{\Omega}|\nabla
u|^{p(.)}dx)^{\frac{1}{p^{+}}},~(\int_{\Omega}|\nabla
u|^{p(.)}dx)^{\frac{1}{p^{-}}}\Big].
\end{split}
\end{equation}

By $(3.2)-(3.3),$ we get
\begin{align}
G'(t)+2C_{1}\min\{G^{\frac{p^{+}}{2}}(t),G^{\frac{p^{-}}{2}}(t)\}\leqslant
2G(t).
\end{align}
Noting that
$2C_{1}\min\{G^{\frac{p^{+}}{2}}(t),G^{\frac{p^{-}}{2}}(t)\}>0,$ we get
\begin{align*}
\|u(.,t)\|_{2}\leqslant\|u_{0}\|_{2}e^{t},
\end{align*}
which implies that
\begin{align}
\min\{G^{\frac{p^{+}}{2}}(t),G^{\frac{p^{-}}{2}}(t)\}\geqslant
\min\{1,\|u_{0}\|^{p^{-}-p^{+}}_{2}\}\Big(G(t)e^{-2t}\Big)^{\frac{p^{+}}{2}}e^{p^{-}t}.
\end{align}
By $(3.4)-(3.5),$ we get
\begin{align}
\frac{d(Ge^{-2t})}{dt}\leqslant-2C_{1}\min\{1,\|u_{0}\|^{p^{-}-p^{+}}_{2}\}
\Big(G(t)e^{-2t}\Big)^{\frac{p^{+}}{2}}e^{(p^{-}-2)t},~~G(0)=\|u_{0}\|^{2}_{2}>0.
\end{align}

$\mbox{Gronwall's}$ inequality implies that the solution of
Inequality $(3.6)$ satisfies the following estimate
\begin{align*}
G(t)\leqslant\Big[\|u_{0}\|^{2-p^{+}}_{2}-K_{1}+K_{1}e^{(p^{-}-2)t}\Big]^{\frac{2}{2-p^{+}}}.
\end{align*}

Secondly, we consider the case when $1<r<2.$  Applying
$\mbox{H\"{o}lder's}$ inequality and Inequality $(3.2)-(3.3)$, we
obtain
\begin{align}
G'(t)+2C_{1}\min\{G^{\frac{p^{+}}{2}},G^{\frac{p^{-}}{2}}\}\leqslant
2|\Omega|^{\frac{2-r}{2}}G^{\frac{r}{2}}(t).
\end{align}
Now, we choose
$A=C_{1}\min\{\|u_{0}\|^{p^{-}-p^{+}}_{2},1\},~B=4|\Omega|^{\frac{2-r}{2}}$.
Let us consider the following problem
\begin{equation}
\begin{cases}
y'(t)=\frac{2-p^{+}}{2}By^{\frac{r-p^{+}}{2-p^{+}}}-\frac{2-p^{+}}{2}A:=F(u(t)),\\
y(0)=\|u_{0}\|^{2-p^{+}}_{2}>0.
\end{cases}
\end{equation}
Due to $2>r>p^{+}$, we may choose sufficiently small $\|u_{0}\|_{2}$
such that $F(u_{0})<0.$ Furthermore, a simple analysis shows that
$F(u(t))$ is decreasing with respect to $t$. Hence, we obtain that
\begin{equation}
F(u(t))\leqslant F(u_{0})<0,~~\forall~t\geqslant0.
\end{equation}
By $(3.8)-(3.9),$ we arrive at the following relations
\begin{align*}
\begin{cases}
0<y(t)\leqslant y(0);\\
y(t)\leqslant y(0)+F(u_{0})t,~~0<t<T_{1}=\frac{y(0)}{-F(u_{0})};\\
y(t)=0,~~t\geqslant T_{1}.
\end{cases}
\end{align*}
It is easy to verify that $y^{\frac{2}{2-p^{+}}}(t)$ is  an
upper-solution of $(3.7)$, then according to comparison principle for ODE in \cite{ACPETERSON},
we get
$$\|u\|^{2}_{2}\leqslant y^{\frac{2}{2-p^{+}}}(t),~~0<t<T_{1}.$$
\end{proof}
When $r<p^{+}<2$, we have
\begin{theorem}
Suppose that $p(x)$satisfies $(1.2)-(1.3).$  If the following
condition holds
\begin{align*}
(H_{10})~\frac{2N}{N+2}<r<p^{-}<p^{+}\leqslant2,
\end{align*}
then the nonnegative solution of Problem $(1.1)$ does not vanish in
finite time for any initial data positively bounded from below.
\end{theorem}
\begin{proof}Let $\lambda_{1}>0$ and $\Phi>0$ be the first eigenvalue and
eigenfunction of the following problem
\begin{equation*}
\begin{cases}
-\mbox{div}(|\nabla\Phi|^{p(x)-2}\nabla\Phi)=\lambda_{1}|\Phi|^{p(x)-2}\Phi,~&x\in\Omega;\\
\Phi=0,~&x\in\partial\Omega.
\end{cases}
\end{equation*}
From~\cite{FANXIANLING}, we know $\Phi\in W^{1,p(x)}_{0}(\Omega)$
satisfies that the following facts
\begin{align*}
\Phi>0,~x\in\Omega,~M=\sup\limits_{x\in\Omega}|\Phi|<\infty.
\end{align*}
For
$0<\varepsilon<\min\{1,\frac{(1+\lambda_{1})^{r-p^{-}}}{eM},\frac{\min\limits_{x\in\Omega}u_{0}}{eM}\}$,
we consider the auxiliary problem
\begin{equation}
\begin{cases}
v_{t}-\mbox{div}(|\nabla v|^{p(x)-2}\nabla v)=\frac{\lambda_{1}v^{r}}{\varepsilon\Phi+\lambda_{1}v},~&(x,t)\in Q_{T},\\
v(x,t)=0,~&(x,t)\in\Gamma_{T}, \\
v(x,0)=u_{0}>0,~&x\in\Omega.
\end{cases}
\end{equation}
It is easy to prove that the solution $u$ of Problem $(1.1)$ is an
upper-solution to Problem $(3.10)$. Using the comparison
principle in \cite{BGWJGAO}, we get $v(x,t)\leqslant
u(x,t),~~(x,t)\in Q_{T}.$

Next, we construct a lower-solution to Problem  $(3.10)$. For any
given $T>0,$ let $w(x,t)=\varepsilon e^{(1-\frac{t}{T})}\Phi,$ then
we have
\begin{align*}
w'(t)\leqslant0,~\lambda_{1}w^{p-1}-\frac{\lambda_{1}w^{r}}{\varepsilon\Phi+\lambda_{1}w}\leqslant0,~(x,t)\in
Q_{T}.
\end{align*}
So, for any nonnegative test-function $\varphi$, we have
\begin{align*}
\iint_{Q_{T}}&[w_{t}\varphi+|\nabla w|^{p(x)-2}\nabla
w\nabla\varphi-
\frac{\lambda_{1}w^{r}}{\varepsilon\Phi+\lambda_{1}w}\varphi]dxdt=
\iint_{Q_{T}}[\lambda_{1}w^{p-1}-
\frac{\lambda_{1}w^{r}}{\varepsilon\Phi+\lambda_{1}w}]\varphi
dxdt\leqslant0.
\end{align*}
Again applying the comparison principle, we get
$$0<w(x,t)\leqslant v(x,t)\leqslant u(x,t),~~(x,t)\in Q_{T}.$$
That is $u$ does not vanish in finite time.
\end{proof}
\begin{remark} For $\frac{2N}{N+2}<p^{-}<r<p^{+}<2,$ what happens
to the solution of Problem $(1.1)$ ? Due to
 technical reasons,
 up to now we can not prove or not whether the solution vanishes and remain positive.
\end{remark}
\begin{theorem} Suppose that $p(x)$satisfies $(1.2)-(1.3).$  If the
following condition holds
\begin{align*}
(H_{11})~1<p^{-}<\frac{2N}{N+2},1<p^{+}<\frac{Np^{-}}{N-p^{-}},r\geqslant2,
\end{align*}
then the bounded nonnegative solution of Problem $(1.1)$ vanishes in
finite time if the initial data is sufficiently small.
\end{theorem}
\begin{proof}
Multiplying $(1.1)$ by $u^{s}(s=\frac{2N-(N+1)p^{-}}{p^{-}})$ and
integrating over $\Omega$, we get
\begin{align}
\frac{1}{s+1}\int_{\Omega}u^{s+1}dx\Big|^{t+h}_{t}+C_{1}\int_{t}^{t+h}\int_{\Omega}|\nabla
u^{\beta}|^{p(x)}dxdt\leqslant C_{2}\int_{t}^{t+h}\int_{\Omega}
u^{s+1}dxdt,
\end{align}
with $\beta=\frac{(2-p^{-})(N-p^{-})}{p^{-}p^{-}}.$

By means of the above inequality and the embedding theorem
$W^{1,p(x)}_{0}(\Omega)\hookrightarrow W^{1,p^{-}}_{0}(\Omega)\hookrightarrow
L^{\frac{Np^{-}}{N-p^{-}}}(\Omega),$ we have
\begin{equation}
\begin{split}
\|u^{\beta}\|_{\frac{Np^{-}}{N-p^{-}}}&\leqslant C\|\nabla
u^{\beta}\|_{p^{-}}\leqslant C\|\nabla u^{\beta}\|_{p(.)}\leqslant
C\max\Big[(\int_{\Omega}|\nabla
u^{\beta}|^{p(.)}dx)^{\frac{1}{p^{+}}},~(\int_{\Omega}|\nabla
u^{\beta}|^{p(.)}dx)^{\frac{1}{p^{-}}}\Big]\\
&\leqslant
C\max[C^{\frac{p^{+}-p^{-}}{p^{+}p^{-}}}_{1}(\|u_{0}||_{2},|\Omega|),1](\int_{\Omega}|\nabla
u^{\beta}|^{p(.)}dx)^{\frac{1}{p^{+}}}\leqslant
C(\int_{\Omega}|\nabla u^{\beta}|^{p(.)}dx)^{\frac{1}{p^{+}}}.
\end{split}
\end{equation}

Dividing $(3.11)$ by $h$ and applying $\rm{Lebesgue}$ differentiation
theorem and Inequality $(3.12)$, we have
\begin{align*}
\frac{1}{s+1}G'(t)+C_{2}G^{p^{+}\frac{N-p^{-}}{Np^{-}}}(t)\leqslant
C_{3}G(t),
\end{align*}
with $G(t)=\int_{\Omega}u^{s+1}dx.$

Recalling $\textrm{Gronwall's}$ inequality, there exists a $T_{3}>0$
such that \begin{align*}
&G(t)\leqslant\Big[G^{\frac{Np^{-}-p^{+}(N-p^{-})}{2Np^{-}}}(0)-\frac{C_{2}}{C_{3}}+\frac{C_{2}}{C_{3}}e^{\frac{C_{3}(2-p^{-})(p^{+}(N-p^{-})-Np^{-})}{p^{-}p^{-}}t
}\Big]^{\frac{Np^{-}}{Np^{-}-p^{+}(N-p^{-})}},~~0<t<T_{3};\\
&G(t)=0,~~t\in[T_{3},\infty),\end{align*} where

$T_{3}=\frac{p^{-}p^{-}}{C_{3}(2-p^{-})(Np^{-}-p^{+}(N-p^{-}))}
\ln\Big[1+\frac{G^{\frac{Np^{-}-p^{+}(N-p^{-})}{2Np^{-}}}(0)}{\frac{C_{2}}{C_{3}}-G^{\frac{Np^{-}-p^{+}(N-p^{-})}{2Np^{-}}}(0)}\Big].$
\end{proof}
\begin{remark}
When $1<p^{-}<\frac{2N}{N+2},\frac{Np^{-}}{N-p^{-}}<p^{+}<2\leqslant
r$, what happens to the solution of Problem $(1.1)$? Due to
 technical reasons,
 up to now we can't prove whether the solution vanishes. But, we
 guess that the solution may vanish for sufficiently small initial data
 and may not vanish for sufficiently large initial data. That is,
 the value of the initial data  plays a role in studying
 the properties of solutions.
\end{remark}


\begin{thebibliography}{99}
\bibitem{EDIBEBN}
Dibenedetto~E.,~Degenerate parabolic equations.~New
York,~Springer-Verlag,~1993
\bibitem{JRPHILIP}
Philip~J.~R.,~N-diffusion,~Austral.~J.~Phy.,~14(1961):1-13.
\bibitem{CATCWJ}
Atkinson~C.,~Jones~C.~W.,~Similarity solutions in some nonlinear
diffusion problems and in boundary-layer flow of a pseudo-plastic
fluid,~Quart.~J.~Mech.~Appl.~Math.~27(1974):193-211.
\bibitem{MRUZICKA}
Ruzicka~M.,~Electrorheological fluids:\,Modelling and Mathematical
Theory.~Lecture Notes in Math.~1748,~Berlin,~Springer,~2000.
\bibitem{LDPHPM}
Diening~L.,~Harjulehto~P.,~H\"{a}st\"{o}~P.,~R\^{u}\v{z}i\v{c}ka~M.,
Lebesgue and Sobolev spaces with variable exponents[M], Lecture
Notes in Mathematics, vol. 2017, Springer-Verlag, Heidelberg, 2011.
\bibitem{JXYCHJ}
Yin~J.~X.,~Jin~C.~H.,~Critical extinction and blow-up exponents for
fast diffusive p-Laplacian with
sources,~Math.~Methods~Appl.~Sci.~30(2007):1147-1167.
\bibitem{LWJWMX}
Liu~W.~J.,~Wang~M.~X.,~Blow-up of solutions for a $p$-Laplacian
equation with positive initial energy,~Acta Appl. Math., 103
(2008):141-146.
\bibitem{SNANTO1}
Antontsev~S.~N.,~Shmarev~S.~I.,~Anisotropic parabolic equations with
variable nonlinearity.~Pub~Math,~2009,~53:~355-399.
\bibitem{SZLWJG}
Lian~S.~Z., Gao~W.~J., Yuan~H. J., Cao~C.~L.,
 Existence of solutions to initial Dirichlet problem of evolution p(x)-Laplace Equations,
 Ann. Inst. H. Poincare Anal. Non Lineaire,29(3)(2012):377-399.
\bibitem{SNANTO2}
Antontsev~S.~N.,~Shmarev~S.~I.,~Blow-up of solutions to parabolic
equations with nonstandard growth
conditions,~J.~Comput.~Appl.~Math.~234(9)(2010):~2633-2645.

\bibitem{FANXIANLING}
Fan~X.~L.,~Remarks on eigenvalue problems involving the
$p(x)$-Laplacian,~J.~Math.~Anal.\\~Appl.~352(1)(2009):85-98.
\bibitem{BGWJGAO}
Guo~B.,Gao~W.J.,~Study of weak solutions for parabolic equations
with nonstandard growth
conditions,~J.~Math.~Anal.~Appl.,~374(2)(2011):~374-384.

\bibitem{JSIMON}
Simon~J.,~Compact sets in the space
$L^{p}(0,T;B)$,~Ann.~Math.~Pura.~Appl.,~4(146)(1987):65-96.

\bibitem{ACPETERSON}
Peterson~A.~C.,~Comparison theorems and Existence theorems for ordinary differential equations,
~J.~Math.~Anal.~Appl.,~55(1976):~773-784.

\end{thebibliography}
\end{document}